\documentclass[a4paper,12pt,reqno]{amsart}

\usepackage{amsfonts}
\usepackage{amsmath}
\usepackage{amssymb}

\usepackage{mathrsfs}
\usepackage{hyperref}
\usepackage[dvipsnames]{xcolor}
\numberwithin{equation}{section}
\theoremstyle{plain}
\newtheorem{thm}{Theorem}[section]

\newtheorem{cor}[thm]{Corollary}

\theoremstyle{definition}
 
\newtheorem{rem}[thm]{Remark}

\def\X{\mathbb X}

\begin{document}
\title[Bilinear Hardy inequalities on metric measure spaces]
{Bilinear Hardy inequalities on metric measure spaces}

\author[Michael Ruzhansky]{Michael Ruzhansky}
\address{
  Michael Ruzhansky:
    \endgraf
    Department of Mathematics: Analysis, Logic and Discrete Mathematics 
        \endgraf
    Ghent University
      \endgraf
    Krijgslaan 281, Building S8 
      \endgraf
B 9000 Ghent, Belgium
      \endgraf
  and 
  \endgraf
    School of Mathematical Sciences
    \endgraf
  Queen Mary University of London
  \endgraf
  Mile End Road, London E1 4NS
    \endgraf
United Kingdom}
  \endgraf
  \email{michael.ruzhansky@ugent.be, m.ruzhansky@qmul.ac.uk}
  
\author[Anjali Shriwastawa]{Anjali Shriwastawa}
\address{Anjali Shriwastawa:
 \endgraf
  DST-Centre for Interdisciplinary Mathematical Sciences  \endgraf
  Banaras Hindu University, Varanasi-221005, India}
\endgraf

\email{anjalisrivastava7077@gmail.com}

\author[Daulti Verma]{Daulti Verma}
\address{
  Daulti Verma:
  \endgraf
    Miranda House College
  \endgraf
  University of Delhi
  \endgraf
  Delhi-110007, India}
   \endgraf
  \email{daulti.verma@mirandahouse.ac.in}

%\date{\today}
\allowdisplaybreaks
\begin{abstract}  In this paper, we discuss the Hardy inequality with bilinear operators on general metric measure spaces.
 We give the characterization of weights for the bilinear Hardy inequality to hold on general metric measure spaces having polar decompositions. We also provide several examples of the results, finding conditions on the weights for integral Hardy inequalities on homogeneous Lie groups, as well as on hyperbolic spaces and more generally on Cartan-Hadamard manifolds. 
 \end{abstract}
\maketitle
\tableofcontents
\section{Introduction}\label{sec1}
The original Hardy inequality  which was given by G. H. Hardy has been intensely studied since 1920, says that: for any $p>1$, if $f$ is a positive measurable function, then the inequality 
\begin{align}\label{h1}
    \int_0^\infty\left(\frac{1}{x}\int_0^x f(y)\,dy\right)^pdx\le \left(\frac{p}{p-1}\right)^p\int_0^\infty f^p(x)\, dx,
\end{align} holds (see \cite{H1}). Again, G. H. Hardy obtained the weighted version of inequality \eqref{h1} (see \cite{H2}): if $f$ is a positive measurable function and if $p>1$ and $\epsilon<p-1,$ then the inequality \begin{align}\label{h2}
    \int_0^\infty \left(\int_0^xf(y)\,dy\right)^p\,x^{\epsilon-p}\,dx\le \left(\frac{p}{p-1-\epsilon}\right)^p\int_0^\infty f^p(x)\,x^{\epsilon}\,dx,
\end{align}
holds. There is a lot of literature available for different forms of Hardy inequalities and it is not an easy task to mention all, but a few of the references are \cite{CG,D,DK,EE,H1,H2,KRS, KKR, KRZ,Kuf,Muck,Ngu,RST1,RST,RS1,RS,RSY,RY}. In \cite{S}, Sinnamon  obtained Hardy inequalities for the case $0<q<1<p<\infty.$   Again, in \cite{SS}, Sinnamon and Stepanov  established Hardy inequalities for the case $0<q<1$ and $p=1.$ For the case $p\le q$, many authors  Bradley \cite{Brad}, Talenti \cite{Talenti} and Muckenhoupt \cite{Muck} obtained Hardy inequalities in different perspectives. In \cite{RV,RV1}, the first author with third author obtained  the Hardy inequalities on metric measure spaces for the case $1< p \leq q< \infty,$  and $0<q<p$, $1<p<\infty$.  In \cite{KRS}, the first author with Kassymov and Suragan obtained a reverse version of the integral Hardy inequality on metric measure spaces with two negative
exponents for the case $q\le p<0.$ Again, In \cite{RST1}, the first and second author with Tiwari  established Hardy inequalities on metric measure space
possessing polar decompositions for the case $p= 1$ and $1\le q<\infty.$ For a discussion concerning polar decompositions in general metric measure spaces see \cite{AR}.
 
Now, we recall a concise summary of the bilinear Hardy inequality in various contexts. The bilinear Hardy inequality has important applications in the study of partial differential equations and harmonic analysis. It allows for the analysis of bilinear operators and their properties, and it provides crucial estimates for the behaviour of solutions to certain PDEs involving bilinear terms. It is worth noting that there are various generalisations and refinements of the bilinear Hardy inequality for different function spaces, domains and dimensions. The specific form and conditions of the inequality may vary depending on the particular setting in which it is applied.

Ca\~{n}estro et al. \cite{CSR} discussed the weighted bilinear Hardy operator for nonnegative measurable functions on $(0,\infty)$ which says that: if $\widetilde{U},\widetilde{V_1}$ and $\widetilde{V_2}$ are weight functions, then the inequality
\begin{align}
\nonumber&\left(\int_0^\infty \left[H_2(F_1,F_2)(y) \right]^q \widetilde{U}(y) dy \right)^{1/q} \leq C \left( \int_0^\infty F_1^{p_1}(y)\widetilde{V_1}(y) dy \right)^{1/p_1}\\& \quad\quad\quad\quad\quad\quad\quad\quad\quad\quad\quad\quad\times\left( \int_0^\infty F_2^{p_2}(y)\widetilde{V_2}(y) dy \right)^{1/p_2},
\end{align} 
holds for $F_1,F_2\geq0$ with different conditions on the indices $p_1,p_2,q.$ Here
\begin{align*}
  H_2(F_1,F_2)(y):=\bigg(\int_0^y F_1(x)dx \bigg)\bigg(\int_0^y F_2(x)dx \bigg),
   \end{align*} is the  bilinear operator.
   This problem has been studied by many authors. 
  Krepela \cite{Krep} proved the boundedness of the bilinear Hardy operator using the iterative technique. The theory of bilinear Hardy inequalities with weights has been developed in many directions. For a general
perspective of the topic, one can refer to papers \cite{garcia,jain,Talenti}.

In this paper, we obtain the bilinear Hardy inequality on general metric measure spaces with polar decompositions, with some examples and applications in the setting of homogeneous Lie groups, hyperbolic spaces, and Cartan-Hadamard manifolds. Also, we give some characterizations of weights for bilinear Hardy inequality. Indeed, we are not assuming any doubling condition on the measure. The main result of our paper is:
\begin{thm}\label{THM:Hardy1}
Let $1<p_1,p_2 <\infty$ and let $0<q<\infty$. Let $\mathbb X $ be a metric measure space with a polar decomposition \eqref{EQ:polar} at a. 
Let $u,v_1,v_2> 0$ be measurable functions  positive a.e in $\mathbb X$  such that $u\in L_{loc}^1(\mathbb X\backslash \{a\})$ and $v_i^{1-p_i'}\in L^1_{loc}(\mathbb X),$ where $i=1,2.$ Let
\begin{align}\label{ball hardy operator}
  & H^B_2(f_1,f_2)(x):=H^Bf_1(x)\cdot H^Bf_2(x)=\int_{B(a,|x|)}f_1(y)dy\int_{B(a,|x|)} f_2(y)\,dy.
\end{align} 
Then the inequality
\begin{align}\label{EQ:Hardy1}
\nonumber&\bigg(\int_\mathbb X\bigg(H_2^B(f_1,f_2)(x)\bigg)^q u(x)dx\bigg)^\frac{1}{q}\le C\bigg\{\int_{\mathbb X} f_1(x)^{p_1}v_1(x)dx\bigg\}^{\frac1{p_1}}\\&\quad\quad\quad\quad\quad\quad\quad\quad\quad\quad\times\bigg\{\int_{\mathbb X} f_2(x)^{p_2}v_2(x)dx\bigg\}^{\frac1{p_2}},
\end{align}
holds for all $f_1,f_2\ge 0,$ if and only if the inequality

\begin{align}\label{EQ:Hardy_B}
\nonumber&\left(\int_0^\infty \left(H_2(F_1,F_2)(\tau) \right)^q \widetilde{U}(\tau) d\tau \right)^{1/q} \leq C \left( \int_0^\infty F_1^{p_1}(\tau)\widetilde{V}_1(\tau) d\tau \right)^{1/p_1}\\& \quad\quad\quad\quad\quad\quad\quad\quad\quad\quad\quad\quad\times\left( \int_0^\infty F_2^{p_2}(\tau)\widetilde{V}_2(\tau) d\tau \right)^{1/p_2},
\end{align} holds for all $F_1,F_2\ge 0,$ where $\widetilde U, \widetilde {V}_1, \widetilde {V}_2$ are the weight functions given by

\begin{align}\label{udef}
\widetilde{U}(\tau):= \int_{\sum_\tau} u(\tau,\omega)\,\lambda(\tau,\omega)\,d\omega,
\end{align} 
and 
\begin{align} \label{vi}
\widetilde{V}_{i}(\tau):= \left(\int_{\sum_\tau}v_i^{1-p_i'}(\tau,\omega)\,\lambda(\tau,\omega)\,d\omega\right)^{1-p_i},
\end{align}
where $i=1,2.$
Moreover, the constants $C$ in \eqref{EQ:Hardy1} and \eqref{EQ:Hardy_B} are same. The function $\lambda(\tau,w)$ comes from the polar decomposition and is defined in \eqref{EQ:polar}.
 \end{thm}
     Apart from Section \ref{sec1} , this paper is divided in four sections.  In Section \ref{sec2}, we will
recall the basics of metric measure spaces, which is the main tool of  this paper. Also, we will recall some brief introduction of homogeneous Lie groups, hyperbolic spaces and  Cartan-Hadamard manifolds, which we will use in applications and examples of the main result of this paper. Section \ref{sec3} is devoted to presenting the proof of the main result of this paper. In Section \ref{sec4}, we will introduce the characterization of weights for the bilinear Hardy operators on metric measure spaces. In Section \ref{sec5}, we will give some applications and examples of our main result of this paper.\\ In this paper, we use $A_1\asymp A_2$ to indicate that $\exists\,C_{1},C_{2}>0$ such that $C_{1}A_1\leq A_2\leq C_{2}A_1$.
 \section{Preliminaries}\label{sec2}
In this section, we give a brief overview on the basics of metric measure spaces. For the applications and examples purposes,  we  also recall homogeneous Lie groups, hyperbolic spaces and Cartan-Hadamard manifolds.   For more details on metric measure spaces  as
well of several functional inequalities on metric measure spaces, we refer to papers \cite{KRS,RV,RV1,RST1} and references therein.
 Let us now recall some basics of metric measure spaces.
 
A metric measure space with {\em polar decomposition} is a metric space $(\mathbb X,d)$ with a Borel measure $dx$, having the  {\em polar decomposition} at a fixed point $a\in{\mathbb X}$:
   \begin{equation}\label{EQ:polar}
   \int_{\mathbb X}g(x)dx= \int_0^{\infty}\int_{\Sigma_r} g(r,\omega) \,\lambda(r,\omega)\,d{\omega_r} \,dr,
   \end{equation}
   where $\lambda \in L^1_{loc},$ $g\in L^1(\mathbb X),$  $(r,\omega)\rightarrow a $ as $r\rightarrow0$.
    Here the set $\Sigma_r=\{x\in \mathbb X:d(x,a)=r\}$ is equipped with measure, which we denote by $d\omega_r$. 
The function $\lambda$ may be dependent on the complete variable $x=(r,\omega)$.\\  The motivation behind condition \eqref{EQ:polar} in the context described above is to provide a decomposition formula for the metric measure space $(\mathbb{X},d)$ that does not rely on a differential structure. In cases where a metric measure space has a differential structure, such as a Riemannian manifold, the function $\lambda(r,\omega)$ can be obtained as the Jacobian of the polar change of coordinates or through the polar decomposition formula. However, the situation being dealt with in this particular context is more general, and there is no assumption of a differential structure on the metric measure space $(\mathbb{X},d)$. Therefore, it becomes useful to introduce a decomposition formula \eqref{EQ:polar} that holds in this more general setting. Importantly, no doubling condition is imposed on the metric measure space. 
Thus, we are considering a wider class of metric measure spaces. Moreover, we provide examples of metric measure spaces that satisfy condition \eqref{EQ:polar} with different expressions for $\lambda(r,\omega)$. Some of these examples include hyperbolic spaces, which are known for their non-doubling volume growth. This demonstrates that the condition \eqref{EQ:polar} can be satisfied in diverse settings, including spaces with non-trivial geometric properties. For a detailed analysis and discussion of condition \eqref{EQ:polar} in the context of general metric measure spaces, we refer to \cite{AR}. \\ We give some examples of metric spaces having polar decomposition.
\begin{itemize}
\item[(i)] {\bf Homogeneous Lie groups $\mathbb{G}$:} For this case, we have $\lambda (r,\omega)= {r}^{Q-1},$ where $Q$ is the homogeneous dimension of the homogeneous Lie groups. Simply, one can notice that for the Euclidean spaces $\mathbb{R}^n,$ we have $\lambda (r,\omega)= {r}^{n-1}.$  For more details on such groups, we refer to \cite{FR,FS,RS}.
\item[(ii)]{\bf Hyperbolic spaces $\mathbb{H}^n$:}  The $n$-dimensional hyperbolic space $\mathbb{H}^n$ is the unique simply connected, $n$-dimensional complete Riemannian manifold with a constant negative sectional curvature equal to $-1$. The unicity means that any two Riemannian manifolds which satisfy these properties are isometric to each other. For the case of hyperbolic spaces, we have $\lambda(r,\omega)=(\sinh {r})^{n-1}.$  For more details on hyperbolics spaces, we refer to \cite{BC,GHL,H}.
\item[(iii)] \textbf{Cartan-Hadamard manifolds:} A simply connected complete Riemannian manifold $(M,g)$ of nonpositive curvature is called a Cartan-Hadamard manifold. Let $K_M\le 0$ be the sectional curvature of the Cartan-Hadamard manifold. Let $\exp_p:T_pM\rightarrow M$ be the exponential map at a fixed point $p$, which is also a diffeomorphism.  If $J(r,w)$ is the density function on $M,$ then   we have the following polar decomposition:
\begin{align}\label{condition:11}
    \int_Mf(y)\,dy=\int_0^\infty\int_{S^{n-1}}
    f(\exp_p(r,\omega))J(r,\omega)r^{n-1} \text{d}\omega\,dr,
\end{align}
with $$\lambda(r,\omega)=J(r, \omega)\,r^{n-1}.$$
We refer to \cite{BC,H} for more explanation about  Cartan-Hadamard manifolds.
 \end{itemize}

 % Let $\Omega \in \mathbb{R}^n$ and let $p,q>1$ Suppose that $f\in L^p(\Omega)$ and $g\in L^q(\Omega)$
 %are two functions defined on $\Omega$
 %Then, the following inequality holds:
%  \begin{align}
%\left|\int_{\Omega}f(x)\,g(x)\right| \le C_{p,q}\|f%\|_{L^p(\Omega)}\,\|g\|_{L^q(\Omega)},
 % \end{align} where $C_{p,q}$ is positive constant depending on $p$ and $q.$ For more details on bilinear Hardy inequalities in different form and different settings we refere to \cite{CSR,Krep}.
\section{Proof of the main result}\label{sec3}
Before proving our main result of the paper, let us introduce some notations, which we are using in our paper. We use $B(a,r)$ for the ball in a metric measure space  $\mathbb X,$ where $a$ is the centre  and   $r$ is the radius of the ball. If $d$ is the metric on $\X$ then
$$B(a,r):= \{x\in\mathbb X : d(x,a)<r\}.$$ 
 Also, we write ${\vert x \vert}_a := d(a,x),$ for a fixed point $a\in {\mathbb X},$ to simplify the notation.\\

%For more details of Hardy operator we refer to \cite{FS}.We also use the polar decomposition on metric measure space $\mathbb{X}$ \ref{EQ:polar} in the Theorem \ref{THM:Hardy1} 
 
Now, we prove our first main result of this paper:

\begin{proof}
  For the complete proof of the Theorem \ref{THM:Hardy1}, we first prove that if the inequality \eqref{EQ:Hardy_B} holds then the inequality  \eqref{EQ:Hardy1} also holds. \\
    Let us suppose that the inequality \eqref{EQ:Hardy_B} holds. Now, for the fixed $f_1$ and $f_2,$ we define 
  \begin{align}\label{F1}
    F_1(s) :=\int_{\sum_s} f_1(s,t)\,\lambda(s,t)\,dt,
  \end{align}
  and \begin{align}\label{G1}
       F_2(s) :=\int_{{\sum_s}} f_2(s,t)\,\lambda(s,t)\,dt,
  \end{align}
where  $\sum_s=\{x\in \mathbb{X}; d(x,a)=s\}$ denotes the level set of $d$ on  $\mathbb{X}$ and $t\in \sum_s.$ 
Again, using $\frac{1}{p_1}+\frac{1}{p_1'}-1=0$ in \eqref{F1}, we get
\begin{align*}
  &F_1(s) =\int_{\sum_s} f_1(s,t)\,\lambda(s,t)\,dt\\&   =\int_{\sum_s}f_1(s,t)\,\lambda(s,t)\,v_1^{\frac{1}{p_1}+\frac{1}{p_1'}-1}(s,t)dt=\int_{\sum_s}f_1(s,t)\,\lambda(s,t)\,v_1^{\frac{1}{p_1}+\frac{1-p_1'}{p'_1}}(s,t)\,dt.
\end{align*}
Applying H$\ddot{\text{o}}$lder inequality and using  $\frac{1}{p_1'(1-p_1)}=-\frac{1}{p_1}$ in the following, we get
\begin{align}\label{F_1s}
\nonumber&F_1(s)=\int_{\sum_s}f_1(s,t)\,\lambda(s,t)\,v_1^{\frac{1}{p_1}+\frac{1-p_1'}{p_1'}}(s,t)\,dt\\&\nonumber\le \left(\int_{\sum_s}f^{p_1}_1(s,t)\,\lambda(s,t)\,v_1(s,t)\,dt \right)^\frac{1}{p_1}\left(\int_{\sum_s}v_1^{1-p_1'}(s,t)\,\lambda(s,t) \,dt\right)^\frac{1}{p_1'}\\&\nonumber=\left(\int_{\sum_s}f^{p_1}_1(s,t)\,\lambda(s,t)\,v_1(s,t)\,dt\right)^\frac{1}{p_1}\left(\widetilde{V}_1(s)\right)^{\frac{1}{p_1'(1-p_1)}}\\&=\left(\int_{\sum_s}f^{p_1}_1(s,t)\,\lambda(s,t)\,v_1(s,t)\,dx \right)^\frac{1}{p_1}\left(\widetilde{V}_1(s)\right)^{-\frac{1}{p_1}}.
\end{align}
Similarly, from \eqref{G1}, we obtain,
\begin{align}\label{F_2s}
 F_2(s)\le  \left(\int_{\sum_s}f^{p_2}_2(s,t)\,\lambda(s,t)\,v_2(s,t)\,dt \right)^\frac{1}{p_2}\left(\widetilde{V}_2(s)\right)^{-\frac{1}{p_2}}.
\end{align}
%Since from \eqref{ball hardy operator}, we have,
%\begin{align*}
   % H^B_2(f_1,f_2):=Hf_1(x)\cdot  Hf_2(x)=\int_{B(0,|x|)}f_1(t) dt\,\int_{B(0,|x|)}f_2(t) dt,
%\end{align*} then
Consider the left hand side of inequality \eqref{EQ:Hardy1}, and using the polar decompositions \eqref{EQ:polar}, we have

\begin{align}
  \nonumber&  \left(\int_\mathbb{X}\left[H^B_2(f_1,f_2)(x)\right]^qu(x)\,dx\right)^\frac{1}{q}\\&\nonumber=\left(\int_{\mathbb{X}}\left(\int_{B(a,|x|)}f_1(y)\,dy\right)^q\left(\int_{B(a,|x|)}f_2(z)\,dz\right)^qu(x)\,dx\right)^\frac{1}{q} \\ \nonumber 
  &=\bigg\{\int_0^\infty\int_{\sum_r}\left(\int_0^r \int_{\sum_{r_1}}f_1(r_1,\omega)\lambda(r_1,\omega)\,d\omega\,dr_1\right)^q\\&\nonumber\qquad\qquad\qquad\qquad\qquad\times \left(\int_0^r\int_{\sum_{r_2}}f_2(r_2,\sigma)\lambda(r_2,\sigma)\,d\sigma\,dr_2\right)^qu(r,\xi)\,\lambda(r,\xi)\,d\xi\,dr\bigg\}^\frac{1}{q}.
\end{align}
Again, by using \eqref{F1} and \eqref{G1}, the last expression is
\begin{align}
%\nonumber&\left(\int_{\mathbb{X}}\left(\int_{B(0,|x|)}f_1(x)\,dx\right)^q\,\left(\int_{B(0,|x|)}f_2(y)\,dy\right)^qu(x)\,dx\right)^\frac{1}{q}\\
&= \bigg\{\int_0^\infty\left(\int_0^rF_1(r_1)\,dr_1\right)^q\left(\int_0^rF_2(r_2)\,dr_2\right)^q\widetilde{U}(r)\,dr\bigg\}^\frac{1}{q}. \nonumber
\end{align}
%Since, by \eqref{hardyoperator},
%we have
%\begin{align}\label{eq2}
 %H_2(F_1,F_2)(r):=HF_1(r)\cdot HF_2(r) =\int_{0}^r F_1(r_1)\,dr_1\,\int_0^r F_2(r_2)\,dr_2.
%\end{align}
%Now using \eqref{eq2} in \eqref{eq11},
%\begin{align}&  \left(\int_0^rF_1(r_1)\,dr_1\right)^q\,\left(\int_0^rF_2(r_2)\,dr_2\right)^q=\left[H_2(F_1,F_2)(r)\right]^q,
%\end{align}
%so that 
\begin{align}
  &\nonumber=\left(\int_0^\infty\left[H_2(F_1,F_2)(r)\right]^q\widetilde{U}(r)\,dr\right).
\end{align} 
As \eqref{EQ:Hardy_B} holds for any $F_1,F_2\ge0,$ by using \eqref{F_1s} and \eqref{F_2s}, we get
\begin{align}
 \nonumber&  \left(\int_0^\infty\left[H_2(F_1,F_2)(r)\right]^q\widetilde{U}(r)\,dr\right)^\frac{1}{q}\\&\nonumber\quad\quad\quad\quad\le
C\left(\int_0^\infty F_1^{p_1}(r)\,\widetilde{V}_1(r)\,dr\right)^\frac{1}{p_1}\left(\int_0^\infty F_2^{p_2}(r)\,\widetilde{V}_2(r)\,dr\right)^\frac{1}{p_2}
\\&\quad \quad \quad \quad \le\nonumber C\left(\int_0^\infty\int_{\sum_r}f_1^{p_1}(r,\omega)v_1(r,\omega)\lambda(r,\omega)dr\,d\omega\right)^\frac{1}{p_1}\\&\nonumber\quad\quad\quad\quad\quad\quad\quad\quad\quad\quad \times\left(\int_0^\infty\int_{\sum_r}f_2^{p_2}(r,\omega)\,v_2(r,\omega)\lambda(r, \omega)dr\,d\omega\right)^\frac{1}{p_2}\\&\quad\quad\quad= C\left(\int_{\mathbb{X}}f_1^{p_1}(x)\,v_1(x)\,dx\right)^\frac{1}{p_1}\left(\int_{\mathbb{X}}f_2^{p_2}(x)\,v_2(x)\,dx\right)^\frac{1}{p_2}, \nonumber
\end{align}
which implies inequality \eqref{EQ:Hardy1}. Also this implies that $C$ in  \eqref{EQ:Hardy1} is $\le$ $C$ in  \eqref{EQ:Hardy_B}.\\
Conversely, suppose that inequality \eqref{EQ:Hardy1} holds. For any  fixed $F_1$ and $F_2$, we set
\begin{align}\label{eq215}
    f_1(t,s_1)=F_1(t)\,v_1^{1-p_1'}(t,s_1)\left(\widetilde{V}_1(t)\right)^{\frac{1}{p_1-1}}
\end{align}
and 
\begin{align}\label{eq216}
  f_2(t,s_2):=F_2(t)\,v_2^{1-p_2'}(t,s_2)\left(\widetilde{V}_2(t)\right)^{\frac{1}{p_2-1}},  
\end{align}
where $t>0,$ and $s_1,s_2\in \sum_t.$ 
Substituting the value of $\widetilde{V}_1(t)$ and $\widetilde{V}_2(t)$ from \eqref{vi} in \eqref{eq215}and \eqref{eq216}, respectively, we obtain
\begin{align}\label{eq217}
  \int_{\sum_t} f_1(t,s_1) \lambda (t,s_1) \,ds_1 = F_1(t)
\end{align}
and \begin{align}\label{eq218}
   \int_{\sum_t} f_2(t,s_2) \lambda (t,s_2) \,ds_2  = F_2(t).
\end{align}
Let us consider the left hand side of the inequality \eqref{EQ:Hardy_B}, by using \eqref{eq217}, \eqref{eq218} and \eqref{udef}, to get
\begin{align}\label{exp1}
& \nonumber \left(\int_0^\infty \left[H_2(F_1,F_2)(\tau) \right]^q \widetilde{U}(\tau) d\tau \right)^{1/q}\\ \nonumber
 &=\left(\int_0^\infty  \left(\int_0^{\tau}F_1(t)\,dt\right)^q\cdot \left(\int_0^\tau F_2(t)\,dt\right)^q \widetilde{U}(\tau)\,d\tau\right)^{1/q}\\ \nonumber
 &=\bigg\{\int_0^\infty  \left(\int_0^{\tau}\int_{\sum_t} f_1(t,\omega)\,\lambda (t,\omega) \,d\omega \,dt\right)^q\\ \nonumber
 & \quad\quad\quad\quad\times \left(\int_0^\tau \int_{\sum_t} f_2(t,\sigma)\,\lambda (t,\sigma) \,d\sigma \,dt\right)^q \int_{\sum_\tau}u(\tau,\xi)\,\lambda(\tau,\xi)\,d\xi\,d\tau \bigg\}^{1/q}\\ \nonumber
 &=\bigg\{\int_{\mathbb{X}}  \left(\int_{B(a,|x|)} f_1(y)\,dy\right)^q\left(\int_{B(a,|x|)} f_2(y)\,dy\right)^q u(x)dx\bigg\}^{1/q}\\
 &=\bigg\{\int_{\mathbb{X}}  \left[H_2^B(f_1,f_2)(x)\right]^q u(x)dx\bigg\}^{1/q}.
\end{align}

Now, by using the inequality \eqref{EQ:Hardy1} and after that using polar decomposition, from the last expression \eqref{exp1}, we get
\begin{align*}
&\bigg\{\int_{\mathbb{X}}  \left[H_2^B(f_1,f_2)(x)\right]^q u(x)dx\bigg\}^{1/q} \le C\bigg\{\int_{\mathbb X}  f_1^{p_1}(x)\,v_1(x)dx\bigg\}^{\frac1{p_1}}  \bigg\{\int_{\mathbb X} f^{p_2}_2(x)\,v_2(x)dx\bigg\}^{\frac1{p_2}}\\&\nonumber = C\bigg\{\int_0^\infty\int_{\sum_t}  f_1^{p_1}(t,s_1)\,\lambda(t,s_1)\,v_1(t,s_1)\,ds_1\,dt\bigg\}^{\frac1{p_1}} \\&\nonumber \qquad\qquad\qquad\times \bigg\{\int_0^\infty \int_{\sum_t}f^{p_2}_2(t,s_2)\,\lambda(t,s_2)\,v_2(t,s_2)ds_2\,dt\bigg\}^{\frac1{p_2}}\\&\nonumber
= C\bigg\{\int_0^\infty\int_{\sum_t}  F^{p_1}_1(t)\,v_1^{p_1(1-p_1')}(t,s_1)\left(\widetilde{V}_1(t)\right)^{\frac{p_1}{p_1-1}}\,\times \lambda(t,s_1)\,v_1(t,s_1)\,dt\,ds_1\bigg\}^{\frac1{p_1}}\\&\nonumber \times  \bigg\{\int_0^\infty \int_{\sum_t}F^{p_2}_2(t)\,v_2^{p_2(1-p_2')}(t,s_2)\left(\widetilde{V}_2(t)\right)^{\frac{p_2}{p_2-1}}\,\lambda(t,s_2)\,v_2(t,s_2)\,dt\,ds_2\bigg\}^{\frac1{p_2}}\\&\nonumber=C\bigg\{\int_0^\infty \int_{\sum_t} v_1(t,s_1)\,\lambda(t,s_1)\,\bigg[F^{p_1}_1(t)  v_1^{p_1(1-p_1')}(t,s_1)\left(\widetilde{V}_1(t)\right)^{\frac{p_1}{p_1-1} }\bigg]ds_1\,dt\bigg\}^{\frac1{p_1}}\\&\nonumber \times \bigg\{\int_0^\infty \int_{\sum_t} v_2(t,s_2)\,\lambda(t,s_2)\,\bigg[F^{p_2}_2(t)  v_2^{p_2(1-p_2')}(t,s_2)\left(\widetilde{V}_1(t)\right)^{\frac{p_2}{p_2-1}}\bigg]ds_2\,dt\bigg\}^{\frac1{p_2}}\\&\nonumber=C\left(\int_0^\infty F_1^{p_1}(t)\left(\int_{\sum_t}v_1^{1-p_1'}(t,s_1)\,\lambda(t,s_1)ds_1\right)\left(\widetilde{V}_1(t)\right)^{\frac{p_1}{p_1-1}} dt\right)^\frac{1}{p_1}\\&\nonumber\quad\quad\quad\quad\times \left(\int_0^\infty F_2^{p_2}(t)\left(\int_{\sum_t}v_2^{1-p_2'}(t,s_2)\,\lambda(t,s_2)ds_2\right)\left(\widetilde{V}_2(t)\right)^{\frac{p_2}{p_2-1}} dt\right)^\frac{1}{p_2} \\&\nonumber=C\bigg\{\int_0^\infty F^{p_1}_1(t)\left(\widetilde{V}_1(t)\right)^{\frac{p_1}{p_1-1}}\,\left(\widetilde{V}_1(t)\right)^{\frac{1}{1-p_1}}\,dt\bigg\}^{\frac1{p_1}}\\&\nonumber \quad\quad\quad\quad\quad\quad \times \bigg\{\int_0^\infty F^{p_2}_2(t)\left(\widetilde{V}_2(t)\right)^{\frac{p_2}{p_2-1}}\,\left(\widetilde{V}_2(t)\right)^{\frac{1}{1-p_2}}\,dt\bigg\}^{\frac1{p_2}}\\&\quad\quad\quad\quad=C\bigg\{\int_0^\infty F^{p_1}_1(t)\,\Tilde{V}_1(t)\,dt\bigg\}^{\frac1{p_1}} \bigg\{\int_0^\infty F^{p_2}_2(t)\,\widetilde{V}_2(t)\,dt\bigg\}^{\frac1{p_2}},
 \end{align*}
which says that \eqref{EQ:Hardy_B} holds, and this also shows that $C$ in  \eqref{EQ:Hardy_B} is $\le$ $C$ in \eqref{EQ:Hardy1}.
  This completes the proof of Theorem \ref{THM:Hardy1}.
\end{proof}

 \section{Weights characterization for different parameters}\label{sec4}
 In this part, the weight characterizations of the inequality \eqref{EQ:Hardy1} is given for different values of $q,p_1,p_2.$
 Let us denote:
\begin{align}\label{EQ:U}
U(x)= { \int_{\mathbb X\backslash{B(a,|x|_a )}} u(y)\,dy}, 
\end{align} 
\begin{align}\label{EQ:V}
V_i(x)= \int_{B(a,|x|_a  )}v_i^{1-p_i'}(y)\,dy,\,\,\, i=1,2 
\end{align}
\begin{align}
U_1(t)= { \int_t^\infty \widetilde{U}(s) \,ds}, 
\end{align}and  
\begin{align}
V_{1i}(t)= \int_0^t {\widetilde{V}_i}^{1-p_i'}(s) \, ds, \    \,i=1,2 \ ,
\end{align}
 where $u, v_i, \widetilde U, \widetilde {V}_i$ are the weight functions, given in \eqref{udef} and \eqref{vi}. Now, we state the following result proved in \cite{CSR}: 
 \begin{thm}\label{THM:Hardy2}
 Let $1<q, p_1,p_2<\infty.$ The inequality \eqref{EQ:Hardy_B} holds for all $F_1,F_2 \geq 0$ and $\widetilde U, \widetilde {V}_1,\widetilde {V}_2$ as weight functions if and only if 
\begin{itemize}
\item[(i)] For the case $1< \max(p_1,p_2) \leq q<\infty,$
\begin{align*}
&\mathcal B_1:= \sup_{0<t<\infty} {U_1} ^\frac{1}{q}(t)   V_{11}^\frac{1}{{p_1}'}(t)  V_{12}^\frac{1}{p_2'}(t)<\infty.
\end{align*}
Also, the best constant $C$ in  \eqref{EQ:Hardy_B} satisfies
\begin{align}
\mathcal B_1 \leq C \leq 8(1+4^q)^\frac 1 q \mathcal B_1.\end{align}
\item[(ii)] For the case $1<p_1\leq q<p_2<\infty,$  $\frac{1}{r_2}=\frac 1 q-\frac 1 p_2,$ \\

 $ \mathcal B_1<\infty,$
\begin{align*}
&\mathcal B_2:= \sup_{0<t<\infty}V_{11}^\frac{1}{{p_1}'}(t)   \bigg( \int_t^\infty  {U_1}^\frac{r_2}{q}(s)\,  {V_{12}}^\frac{r_2}{q'}(s)\,{\widetilde{V}_2}^{1-p_2'}(s)\,ds \bigg)^\frac{1}{r_2}<\infty.
\end{align*}
Also, the best constant $C$ in  \eqref{EQ:Hardy_B} satisfies\begin{align}
    \max\bigg\{\mathcal B_1,q^\frac1 q\left(\frac{qp_2'}{r_2}\right)^\frac{1}{q'} \mathcal B_2\bigg\} \leq 8\left(\mathcal B_1+q^{\frac{1}{q}}\left(p_2'\right)^\frac{1}{q'}\mathcal B_2 \right).
\end{align}
\item[(iii)] For the case $1<q< \min(p_1,p_2) <\infty,$  $\frac 1 q\leq \frac 1 p_1 +\frac 1 p_2,$    $\frac 1r_i=\frac 1 q-\frac 1 p_i,$ $i=1,2,$  \\

$ \mathcal B_1<\infty,\,
\mathcal B_2<\infty,$ 
\begin{align*}
&\mathcal B_3:= \sup_{0<t<\infty}V_{12}^\frac{1}{p_2'}(t) \bigg( \int_t^\infty U_1^\frac{r_1}{q}(s)  V_{11}^\frac{r_1}{q'}(s){\Tilde{V}_1}^{1-p_1'}(s)ds\bigg)^\frac{1}{r_1}<\infty.
\end{align*}
Also, the best constant $C$ in  \eqref{EQ:Hardy_B} satisfies\begin{align}
   \nonumber &\max\bigg\{\mathcal B_1,q^\frac1 q\left(\frac{qp_2'}{r_2}\right)^\frac{1}{q'} \mathcal B_2,q^\frac1 q\left(\frac{qp_1'}{r_1}\right)^\frac{1}{q'} \mathcal B_3\bigg\} \leq C \\&\quad\quad\quad\quad\quad\quad\quad\quad\quad\quad\leq 8\left(8\mathcal B_3+4\left(\frac{p_1'}{r_1}\right)^\frac{1}{r_1}\mathcal B_1+q^{\frac{1}{q}}\left(p_2'\right)^\frac{1}{q'}\mathcal B_2 \right).
\end{align}
\item[(iv)] For the case $1<q< \min(p_1,p_2) <\infty,$  $\frac 1 q > \frac 1 p_1+\frac 1 p_2,$ $\frac 1 k =\frac 1 q-\frac 1 p_1-\frac 1 p_2,$  \\
$\frac 1 r_i=\frac 1 q-\frac 1 p_i,$ $i=1,2,$
\begin{align*}
&\mathcal B_4:=  \bigg\{ \int_0^\infty    U_1^{\frac{k}{p_1}+\frac{k}{p_2}}(t)V_{11}^\frac{k}{p_1'}(t)V_{12}^\frac{k}{p_2'}(t)\widetilde{U}(t)dt \bigg \}^\frac{1}{k}
 <\infty.
\end{align*}
\begin{align*}
&\mathcal B_5:=  \bigg\{ \int_0^\infty  \bigg( \int_t^\infty   U_1^\frac{r_2}{q}(s)V_{12}^\frac{r_2}{q'}(s)
 {\widetilde{V}_2}^{1-p_2'}(s)ds\bigg)^{k/r_2}
  V_{11}^\frac{k}{r_2'}(t){\widetilde{V}_1}^{1-p_1'}(t) dt \bigg\}^{1/k} <\infty.
\end{align*}
and
\begin{align*}
&\mathcal B_6:=   \bigg\{ \int_0^\infty  \bigg( \int_t^\infty   U_1^\frac{r_1}{q}(s)V_{11}^\frac{r_1}{q'}(s)\,
 {\widetilde{V}_1}^{1-p_1'}(s)ds\bigg)^{k/r_2}
  V_{12}^\frac{k}{r_1'}(t)\,{\widetilde{V}_2}^{1-p_2'}(t) dt \bigg\}^{1/k}<\infty.
\end{align*}
\end{itemize}
 \end{thm}
Our result in this section is that the conditions on the weight functions are found for the inequality \eqref{EQ:Hardy1} to hold true which reads out as follows:
\begin{thm}\label{THM:Hardy3}
Let $0<q<\infty$, $1<p_1,p_2<\infty,$ and $u$, $v_1$, $v_2$ are weight functions positive a.e. in $\mathbb{X}$ such that $u\in L_{loc}^1(\mathbb{X}\backslash\{a\})$ and $v_i^{1-p_i'}\in L_{loc}^1(\mathbb{X}),$ where $i=1,2.$ The inequality \eqref{EQ:Hardy1} holds for all $f_1,f_2 \geq 0$ if and only if 
\begin{itemize}
\item[(i)] For $1< \max(p_1,p_2) \leq q<\infty,$
\begin{align*}
&\mathcal D_1:= \sup_{x\not=a} \bigg\{U^\frac{1}{q}(x) V_1^\frac{1}{p_1'}(x)V_2^\frac{1}{p_2'}(x)\bigg\} <\infty.
\end{align*}
\item[(ii)]For $q \geq p_1$ and $q<p_2$ with $p_1,p_2,q>1$ and $\frac{1}{r_2}=\frac{1}{q}-\frac{1}{p_2},$ 
\begin{align*}
&\mathcal D_1 <\infty ,
\end{align*}
\begin{align*}
\mathcal D_2:=  \sup_{x\not=a} V_1^\frac{1}{p_1'}(x)\bigg\{\int_{\mathbb X\backslash{B(a,|x|_a )}}U^\frac{r_2}{q}(y) V_2^\frac{r_2}{q'}(y) {v_2}^{1-p_2'}(y)dy\bigg\}^\frac{1}{r_2}<\infty.
\end{align*}
\item[(iii)]For $q < p_1$ and $q<p_2$ with $p_1,p_2,q>1$ and $\frac{1}{q}\leq\frac{1}{p_1}+\frac{1}{p_2},$ and $\frac{1}{r_i}=\frac{1}{q}-\frac{1}{p_i}$  where $i=1,2,$
\begin{align*}
&\mathcal D_1<\infty, \, \mathcal D_2<\infty , \\
&{\mathcal D_3}:=  \sup_{x\not=a} V_2^\frac{1}{p_2'}(x)\bigg\{\int_{\mathbb X\backslash{B(a,|x|_a )}}U^\frac{r_1}{q}(y) V_1^\frac{r_1}{q'}(y) {v_1}^{1-p_1'}(y)dy\bigg\}^\frac{1}{r_1}<\infty.
\end{align*}
\item[(iv)]For $q < p_1$ and $q<p_2$ with $p_1,p_2,q>1$ and $\frac{1}{q}>\frac{1}{p_1}+\frac{1}{p_2},$ $\frac{1}{r_i}=\frac{1}{q}-\frac{1}{p_i},\, i=1,2,$ and $\frac{1}{k}=\frac{1}{q}-\frac{1}{p_1}-\frac{1}{p_2},$ 
\begin{align*}
&{\mathcal D_4}:=   \bigg\{\int_{\mathbb X} U^{\frac{k}{p_1}+\frac{k}{p_2}}(y) V_1^\frac{k}{p_1'}(y) V_2^\frac{k}{p_2'}(y)u(y)dy \bigg \}^\frac{1}{k}<\infty,
\end{align*}
\begin{align*}
&{\mathcal D_5}:=   \bigg\{\int_{\mathbb X} \bigg\{\int_{\mathbb X\backslash{B(a,|x|_a )}}U^\frac{r_2}{q}(y) V_2^\frac{r_2}{q'}(y) {v_2}^{1-p_2'}(y)dy\bigg\}^\frac{k}{r_2}\\&\quad\quad\quad \quad\quad \times V_1^\frac{k}{r_2'}(x){v_1}^{1-p_1'}(x)dx\bigg)^\frac{1}{k}<\infty,
\end{align*}
and
\begin{align*}
&{\mathcal D_6}:=  \bigg\{\int_{\mathbb X} \bigg\{\int_{\mathbb X\backslash{B(a,|x|_a )}}U^\frac{r_1}{q}(y) V_1^\frac{r_1}{q'}(y) {v_1}^{1-p_1'}(y)dy\bigg\}^\frac{k}{r_1}\\&\quad\quad\quad\quad\quad\times V_2^\frac{k}{r_1'}(x){v_2}^{1-p_2'}(x)dx\bigg\}^\frac{1}{k}<\infty.
\end{align*}
\end{itemize}
\end{thm}
\begin{proof}
First we will show that the conditions $\mathcal D_i$ are equivalent to the conditions $B_i$ in Theorem \ref{THM:Hardy2} for $i=1$ to $6.$ Then the result follows from Theorem \ref{THM:Hardy2}. \\
Let us first show that $\mathcal D_1=\mathcal B_1.$
\begin{align*}
&  \mathcal D_1:= \sup_{x\not=a} \bigg\{U^\frac{1}{q}(x) V_1^\frac{1}{p_1'}(x)V_2^\frac{1}{p_2'}(x)\bigg\}\\ \nonumber 
&= \sup_{x\not=a}\bigg( \int_{\mathbb X\backslash{B(a,|x|_a )}} u(y) dy\bigg)^{1/q}\bigg(\int_{B(a,|x|_a  )}v_1^{1-p_1'}(y)dy\bigg)^{1/p_1'}\bigg(\int_{B(a,|x|_a  )}v_2^{1-p_2'}(y)dy\bigg)^{1/p_2'}\\ \nonumber
&=\sup_{x\not=a} \bigg(\int_{\vert x \vert_a}^{\infty}\int_{\Sigma_r}u(r,\omega)\lambda(r,\omega)d\omega_r dr\bigg)^{1/q}\bigg(\int_0^{\vert x \vert_a}\int_{\Sigma_r}v_1^{1-p_1'}(r,\omega) \lambda(r,\omega)d\omega_r dr\bigg)^{1/p_1'}\\ \nonumber 
&\quad\quad\quad\quad\quad\times \bigg(\int_0^{\vert x \vert_a}\int_{\Sigma_r}v_2^{1-p_2'}(r,\omega) \lambda(r,\sigma)d\omega_r dr\bigg)^{1/p_2'}\\ \nonumber
&=\sup_{x\not=a}\bigg(\int_{\vert x \vert_a}^{\infty}\widetilde{U}(r)dr \bigg)^{1/q}\bigg(\int_0^{\vert x \vert_a} \widetilde{V_1}^{1/(1-p_1)}(r)dr\bigg)^{1/p_1'} \bigg(\int_0^{\vert x \vert_a} \widetilde{V_2}^{1/(1-p_2)}(r)dr\bigg)^{1/p_2'}\\ \nonumber
&=\sup_{0<{\vert x \vert}_a<\infty}\bigg(\int_{\vert x \vert_a}^{\infty}\widetilde{U}(r)dr \bigg)^{1/q}\bigg(\int_0^{\vert x \vert_a} \widetilde{V_1}^{(1-p_1')}(r)dr\bigg)^{1/p_1'} \bigg(\int_0^{\vert x \vert_a} \widetilde{V_2}^{(1-p_2')}(r)dr\bigg)^{1/p_2'}\\ \nonumber
&=\mathcal B_1.
\end{align*}
Now, we prove that $\mathcal D_2=\mathcal B_2.$ We have
\begin{align*}
&\mathcal D_2:=  \sup_{x\not=a} {V}_1^\frac{1}{p_1'}(x)\bigg\{\int_{\mathbb X\backslash{B(a,|x|_a )}}{U}^\frac{r_2}{q}(y) {V}_2^\frac{r_2}{q'}(y) {v_2}^{1-p_2'}(y)dy\bigg\}^\frac{1}{r_2}\\ \nonumber
&=\sup_{x\not=a} \bigg(\int_{B(a,|x|_a  )}v_1^{1-p_1'}(y)dy\bigg)^{1/p_1'}\bigg\{ \int_{\mathbb X\backslash{B(a,|x|_a )}}\bigg( \int_{\mathbb X\backslash{B(a,|y|_a )}} u(z) dz\bigg)^{r_2/q}\\ \nonumber
& \quad\times\bigg(\int_{B(a,|y|_a  )}v_2^{1-p_2'}(z)dz\bigg)^{r_2/q'}{v_2}^{1-p_2'}(y)dy\bigg\}^{1/r_2}\\ \nonumber
&=\sup_{x\not=a}\bigg(\int_0^{\vert x \vert_a}\int_{\Sigma_r}v_1^{1-p_1'}(r,\omega) \lambda(r,\sigma)d\omega_r dr\bigg)^{1/p_1'}\bigg(\int_{\vert x \vert_a}^{\infty}\int_{\Sigma_t}\bigg(\int_t^{\infty}\int_{\Sigma_r}u(r,\omega)\lambda(r,\omega)d\omega_r dr\bigg)^{r_2/q}\\ \nonumber
& \quad \quad \times\bigg(\int_0^t \int_{\Sigma_r}v_2^{1-p_2'}(r,\omega) \lambda(r,\sigma)d\omega_r dr\bigg)^{r_2/q'}{v_2}^{1-p_2'}(t,\rho)\lambda(t,\rho)d\rho_t dt\bigg)^{1/r_2} \\ \nonumber
&=\sup_{x\not=a}\bigg(\int_0^{\vert x \vert_a}\widetilde{V_1}^{1/(1-p_1)}(r)dr \bigg)^{1/p_1'}\bigg(\int_{\vert x \vert_a}^{\infty}\bigg(\int_t^{\infty} \widetilde{U}(r)dr
\bigg)^{r_2/q}\\ \nonumber
& \quad \quad \quad \times\bigg(\int_0^t \widetilde{V_2}^{1/(1-p_2)}(r)dr \bigg)^{r_2/q'} {\widetilde{V}_2}^{1-p_2'}(t)dt \bigg)^{1/r_2}\\ \nonumber
&=\sup_{0<{\vert x \vert}_a<\infty}\bigg(\int_0^{\vert x \vert_a}\widetilde{V_1}^{(1-p_1')}(r)dr \bigg)^{1/p_1'}\bigg(\int_{\vert x \vert_a}^{\infty}\bigg(\int_t^{\infty} \widetilde{U}(r)dr
\bigg)^{r_2/q}\\ \nonumber
& \quad \quad \quad \times\bigg(\int_0^t \widetilde{V_2}^{(1-p_2')}(r)dr \bigg)^{r_2/q'} {\widetilde{V}_2}^{1-p_2'}(t)dt \bigg)^{1/r_2}\\ \nonumber
&=\mathcal B_2.
\end{align*}
Next we prove $\mathcal D_3=\mathcal B_3,$ and the result for
$\mathcal D_4=\mathcal B_4$ is similar.
\begin{align*}
&\mathcal D_3:=  \sup_{x\not=a} {V}_2^\frac{1}{p_2'}(x)\bigg\{\int_{\mathbb X\backslash{B(a,|x|_a )}}{U}^\frac{r_1}{q}(y) {V}_1^\frac{r_1}{q'}(y) {v_1}^{1-p_1'}(y)dy\bigg\}^\frac{1}{r_1}\\ \nonumber
&=\sup_{x\not=a} \bigg(\int_{B(a,|x|_a  )}v_2^{1-p_2'}(y)dy\bigg)^{1/p_2'}\bigg\{ \int_{\mathbb X\backslash{B(a,|x|_a )}}\bigg( \int_{\mathbb X\backslash{B(a,|y|_a )}} u(z) dz\bigg)^{r_1/q}\\ \nonumber
& \quad\times\bigg(\int_{B(a,|y|_a  )}v_1^{1-p_1'}(z)dz\bigg)^{r_1/q'}{v_1}^{1-p_1'}(y)dy\bigg\}^{1/r_1}\\ \nonumber
&=\sup_{x\not=a}\bigg(\int_0^{\vert x \vert_a}\int_{\Sigma_r}v_2^{1-p_2'}(r,\omega) \lambda(r,\sigma)d\omega_r dr\bigg)^{1/p_2'}\bigg(\int_{\vert x \vert_a}^{\infty}\int_{\Sigma_t}\bigg(\int_t^{\infty}\int_{\Sigma_r}u(r,\omega)\lambda(r,\omega)d\omega_r dr\bigg)^{r_1/q}\\ \nonumber
& \quad \quad \times\bigg(\int_0^t \int_{\Sigma_r}v_1^{1-p_1'}(r,\omega) \lambda(r,\sigma)d\omega_r dr\bigg)^{r_1/q'}{v_1}^{1-p_1'}(t,\rho)\lambda(t,\rho)d\rho_t dt\bigg)^{1/r_1} \\ \nonumber
&=\sup_{x\not=a}\bigg(\int_0^{\vert x \vert_a}\widetilde{V_2}^{1/(1-p_2)}(r)dr \bigg)^{1/p_2'}\bigg(\int_{\vert x \vert_a}^{\infty}\bigg(\int_t^{\infty} \widetilde{U}(r)dr
\bigg)^{r_1/q}\\ \nonumber
& \quad \quad \quad \times\bigg(\int_0^t \widetilde{V_1}^{1/(1-p_1)}(r)dr \bigg)^{r_1/q'} 
\widetilde{V_1}^{1/(1-p_1)}(t)dt \bigg)^{1/r_1}\\ \nonumber
&=\sup_{0<{\vert x \vert}_a<\infty}\bigg(\int_0^{\vert x \vert_a}\widetilde{V_2}^{(1-p_2')}(r)dr \bigg)^{1/p_2'}\bigg(\int_{\vert x \vert_a}^{\infty}\bigg(\int_t^{\infty} \widetilde{U}(r)dr
\bigg)^{r_1/q}\\ \nonumber
& \quad \quad \quad \times\bigg(\int_0^t \widetilde{V_1}^{(1-p_1')}(r)dr \bigg)^{r_1/q'} \widetilde{V_1}^{(1-p_1')}(t)dt \bigg)^{1/r_1}\\ \nonumber
&=\mathcal B_3.
\end{align*}
Finally, we prove that $\mathcal D_5=\mathcal B_5,$ and the result for
$\mathcal D_6=\mathcal B_6$ follows similar lines.
\begin{align*}
&{\mathcal D_5}:=  \bigg\{\int_{\mathbb X} \bigg\{\int_{\mathbb X\backslash{B(a,|x|_a )}}U^\frac{r_2}{q}(y) V_2^\frac{r_2}{q'}(y) {v_2}^{1-p_2'}(y)dy\bigg\}^\frac{k}{r_2}\\&\quad\quad\quad \quad\quad V_1^\frac{k}{r_2'}(x){v_1}^{1-p_1'}(x)dx\bigg\}^\frac{1}{k}\\ \nonumber\
&=\bigg(\int_0^\infty \int_{\Sigma_t} \bigg(\int_s^{\infty} \int_{\Sigma_r}\bigg(\int_t^{\infty}\int_{\Sigma_r}u(r,\omega)\lambda(r,\omega)d\omega_r dr\bigg)^{r_2/q}\\ \nonumber
& \quad \quad \times\bigg(\int_0^t \int_{\Sigma_r}v_2^{1-p_2'}(r,\omega) \lambda(r,\sigma)d\omega_r dr\bigg)^{r_2/q'}{v_2}^{1-p_2'}(t,\rho)\lambda(t,\rho)d\rho_t dt\bigg)^{k/r_2}\\ \nonumber
&\quad\times\bigg(\int_0^s \int_{\Sigma_r}v_1^{1-p_1'}(r,\omega) \lambda(r,\omega)d\omega_r dr\bigg)^{k/r_2'}{v_1}^{1-p_1'}(s,\rho)\lambda(s,\rho)d\rho_s ds\bigg)^{1/k}\\ \nonumber
&=\bigg(\int_0^\infty \bigg(\int_s^{\infty}\bigg(\int_t^{\infty} \widetilde{U}(r)dr\bigg)^{r_2/q}\bigg(\int_0^t \widetilde{V_2}^{1/(1-p_2)}(r)dr\bigg)^{r_2/q'}\widetilde{V_2}^{1/(1-p_2)}(r)dt\bigg)^{k/r_2}\\ \nonumber
& \quad\times \bigg(\int_0^s \widetilde{V_1}^{1/(1-p_1)}(r)dr\bigg)^{k/r_2'}\widetilde{V_1}^{1/(1-p_1)}(s)ds\bigg)^{1/k}\\ \nonumber
&=\mathcal B_5.
\end{align*}
Now, Thoerem \ref{THM:Hardy3} follows from Theorem \ref{THM:Hardy2}.
\end{proof}
\section{Applications and examples}\label{sec5}
 In this section, we give several examples of applications of our results to characterize the weights $u$, $v_1$ and $v_2$ for bilinear Hardy inequalities to hold. 
 \subsection{Homogeneous Lie groups}  Let $\mathbb X=\mathbb G$ be a homogeneous Lie group. For more information on these groups, we refer to the references \cite{FR, FS} and \cite{RS}. In this case, we have that $ \lambda (r,\omega)= {r}^{Q-1}$, with $Q$ being the homogeneous dimension of the group, satisfies \eqref{EQ:polar}. We also fix
 $a=0$  to be the identity element of the group $\mathbb G$ which does not cause any loss of generality. The notation is being made simpler by denoting $\vert x \vert_a$ by $\vert x \vert$. We observe that this goes well with the notation for the quasi-norm $|\cdot|$ on a homogeneous Lie group $\mathbb G$.
Let us select the power weights as
$$u(x)= {\vert x \vert}^\alpha, v_1(x)= {\vert x \vert }^{\beta_1},\textrm{ and }v_2(x)= {\vert x \vert }^{\beta_2}.$$
Then the inequality \eqref{EQ:Hardy1} holds for $1< \max(p_1,p_2) \leq q<\infty$ if and only if 
\begin{align}
    \nonumber&{\mathcal{D}_1}=\sup_{r>0} \bigg(\displaystyle \sigma\int_{r}^\infty {\rho}^\alpha {\rho}^{Q-1}d\rho \bigg)^\frac{1}{q} \bigg(\displaystyle \sigma \int_{0}^r{\rho}^{{\beta_1}(1-{p'_1})} {\rho}^{Q-1}d\rho\bigg)^\frac{1}{{p'_1}}\\&
\quad\quad\quad\quad\quad\times\bigg(\displaystyle \sigma \int_{0}^r{\rho}^{{\beta_2}(1-{p'_2})} {\rho}^{Q-1}d\rho\bigg)^\frac{1}{{p'_2}}<\infty,
\end{align}
where $\sigma$ is the area of the unit sphere in $\mathbb G$ with respect to the quasi-norm $|\cdot|.$ For this supremum to be well-defined we need to have  $\alpha+Q<0,$ ${\beta_1}(1-{p'_1})+Q>0$ and ${\beta_2}(1-{p'_2})+Q>0.$ Then we have
\begin{align}
\nonumber&{\mathcal {D}_1}=\sigma^{(\frac1q+\frac{1}{p_1'}+\frac{1}{p_2'})}\sup_{r>0}
\bigg(\displaystyle\int_{r}^\infty{\rho}^{\alpha+Q-1}d\rho\bigg)^\frac{1}{q}\bigg(\displaystyle\int_0^{r} {\rho}^{\beta_1(1-p_1')+Q-1}d\rho \bigg)^\frac{1}{p_1'}\\&\quad\quad\quad\quad\quad\nonumber\times\bigg(\displaystyle\int_0^{r} {\rho}^{{\beta_2}(1-p_2')+Q-1}d\rho \bigg)^\frac{1}{p_2'} \\&
= \sigma^{(\frac1q+\frac{1}{p'}+\frac{1}{p_2'})}
\sup_{r>0} \frac {{r}^ \frac{\alpha+Q}{q}}{{|\alpha+Q|}^\frac{1}{q}}
\frac{ {r}^\frac{\beta_1(1-p_1')+Q} {p_1'}}{{(\beta_1(1-p_1')+Q)}^\frac{1}{p_1'}}\frac{ {r}^\frac{{\beta_2}(1-p_2')+Q} {p_2'}}{{({\beta_2}(1-p_2')+Q)}^\frac{1}{p_2'}},
\end{align} 
which is finite if and only if the power of $r$ is zero.
Therefore, we obtain the following result:

\begin{cor}\label{COR:hom}
Let $\mathbb G$ be a homogeneous Lie group with a quasi-norm $|\cdot|$ and homogeneous dimension $Q$. Let $1< \max(p_1,p_2) \leq q<\infty$ and let $\alpha,\beta_1,\beta_2 \in\mathbb R$. Then the inequality
\begin{align}\label{EQ:Hom}
\nonumber&\bigg(\int_\mathbb G\bigg(H_2^B(f_1,f_2)\bigg)^q {\vert x \vert}^{\alpha}dx\bigg)^\frac{1}{q}\le  C\bigg\{\int_{\mathbb G} {\vert f_1(x) \vert}^{p_1}{\vert x \vert}^{\beta_1}dx\bigg\}^{\frac1{p_1}}\\&\quad\quad\quad\quad\quad\quad \quad\quad\quad\quad\times \bigg\{\int_{\mathbb G} {\vert f_2(x) \vert}^{p_2}{\vert x \vert}^{\beta_2}dx\bigg\}^{\frac1{p_2}}
\end{align}
holds for all measurable functions $f:\mathbb G\to{\mathbb C}$  if and only if 
$$\alpha+Q<0,\,\beta_1(1-p_1')+Q>0,\, {\beta_2}(1-p_2')+Q>0$$ and 
$$\frac{\alpha+Q}{q}+\frac{\beta_1(1-p_1')+Q} {p_1'}+\frac{{\beta_2}(1-p_2')+Q} {p_2'}=0.$$
\end{cor}

\vspace{0.5cm}

 Consider now the condition $(ii)$ of Theorem \ref{THM:Hardy3} with the power weights $$u(x)= {\vert x \vert}^\alpha, v_1(x)= {\vert x \vert }^{\beta_1},\textrm{ and }v_2(x)= {\vert x \vert }^{\beta_2}.$$Let  $q\ge p_1,$ $q<p_2,$ with  $p_1,p_2,q>1$ and $\frac{1}{r_2}=\frac{1}{q}-\frac{1}{p_2},$
then the inequality \eqref{EQ:Hardy1} holds for  $\mathcal{D}_1<\infty$ and 
\begin{align}
\nonumber&\mathcal{D}_2=\sigma^{\frac{1}{p'_1}+\frac{1}{q}+\frac{1}{q'}+\frac{1}{r_2}}\sup_{r>0}\left(\int_0^r \rho^{\beta_1(1-p'_1)+Q-1}d\rho\right)^{\frac{1}{p'_1}}\\&\left(\int_r^\infty \left(\int_\rho^\infty s^{\alpha+Q-1}ds\right)^{\frac{r_2}{q}}\left(\int_0^\rho s^{\beta_2(1-p'_2)+Q-1}ds\right)^{\frac{r_2}{q'}}\rho^{\beta_2(1-p'_2)+Q-1}d\rho\right)^\frac{1}{r_2}<\infty,
\end{align}
where $\sigma$ is the area of the unit sphere in $\mathbb G$ with respect to the quasi-norm $|\cdot|.$ For this supremum to be well-defined we need to have $\beta_1(1-p'_1)+Q>0,$ $\alpha+Q<0,$ $\beta_2(1-p'_2)+Q>0$ and $\frac{(\alpha+Q)r_2}{q}+\frac{(\beta_2(1-p'_2)+Q)r_2}{q'}+\beta_2(1-p'_2)+Q<0.$
 Then we have
\begin{align}
\nonumber&\mathcal{D}_2=\sigma^{\frac{1}{p'_1}+\frac{1}{q}+\frac{1}{q'}+\frac{1}{r_2}}\sup_{r>0}\frac{r^{\frac{\beta_1(1-p'_1)+Q}{p'_1}}}{(\beta_1(1-p'_1)+Q)^{\frac{1}{p'_1}}}\\&\left(\int_r^\infty\frac{\rho^{\frac{(\alpha+Q)r_2}{q}}}{|\alpha+Q|^{\frac{r_2}{q}}}\,\,\frac{\rho^{\frac{(\beta_2(1-p'_2)+Q)r_2}{q'}}}{(\beta_2(1-p'_2)+Q)^{\frac{r_2}{q'}}}\,\,\rho^{\beta_2(1-p'_2)+Q-1}\,d\rho\right)^\frac{1}{r_2}\nonumber,
\end{align}
which gives
\begin{align}
\nonumber&\mathcal{D}_2=\sigma^{\frac{1}{p'_1}+\frac{1}{q}+\frac{1}{q'}+\frac{1}{r_2}}\sup_{r>0}\frac{r^{\frac{\beta_1(1-p'_1)+Q}{p'_1}}}{(\beta_1(1-p'_1)+Q)^{\frac{1}{p'_1}}}\times\\&\frac{1}{|\alpha+Q|^{\frac{1}{q}}\,(\beta_2(1-p'_2)+Q)^{\frac{1}{q'}}}\,\frac{r^{\frac{(\alpha+Q)}{q}+\frac{\beta_2(1-p'_2)+Q}{q'}+\frac{\beta_2(1-p'_2)+Q}{r_2}}}{\left[\frac{(\alpha+Q)r_2}{q}+\frac{(\beta_2(1-p'_2)+Q)r_2}{q'}+(\beta_2(1-p'_2)+Q)\right]^\frac{1}{r_2}}\nonumber\\&\nonumber=\sigma^{\frac{1}{p'_1}+\frac{1}{q}+\frac{1}{q'}+\frac{1}{r_2}}\frac{1}{|\alpha+Q|^{\frac{1}{q}}\,(\beta_2(1-p'_2)+Q)^{\frac{1}{q'}}}\sup_{r>0}\frac{r^{\frac{(\alpha+Q)}{q}+\frac{\beta_2(1-p'_2)+Q}{q'}+\frac{\beta_2(1-p'_2)+Q}{r_2}+\frac{\beta_1(1-p'_1)+Q}{p'_1}}}{\left[\frac{(\alpha+Q)r_2}{q}+\frac{(\beta_2(1-p'_2)+Q)r_2}{q'}+(\beta_2(1-p'_2)+Q)\right]^\frac{1}{r_2}}
\end{align}
which is finite if and only if the power of $r$ is zero. Therefore, \begin{align}\label{5.5}
\frac{(\alpha+Q)}{q}+\frac{\beta_2(1-p'_2)+Q}{q'}+\frac{\beta_2(1-p'_2)+Q}{r_2}+\frac{\beta_1(1-p'_1)+Q}{p'_1}=0,
\end{align} with the condition $\mathcal{D}_1$ from Corollary \ref{COR:hom},
\begin{align}\label{5.6}
\frac{\alpha+Q}{q}+\frac{\beta_1(1-p_1')+Q} {p_1'}+\frac{{\beta_2}(1-p_2')+Q} {p_2'}=0.\end{align}
Solving \eqref{5.5} and \eqref{5.6}, we have 
\begin{align}
\frac{1}{q'}+\frac{1}{r_2}=\frac{1}{p'_2}\implies\frac{1}{r_2}=\frac{1}{q}-\frac{1}{p_2},
\end{align}
which is our given condition of $(ii).$
Therefore, we obtain the following result:
\begin{cor}
	Let $\mathbb G$ be a homogeneous Lie group with a quasi-norm $|\cdot|$ and homogeneous dimension $Q$. Let $q\ge p_1$ and $q<p_2$ with $p_1,p_2,q>1$ and $\frac{1}{r_2}=\frac{1}{q}-\frac{1}{p_2,}$ and let $\alpha,\beta_1,\beta_2 \in\mathbb R$. Then the inequality
	\begin{align}\label{EQ:Hom}
	\nonumber&\bigg(\int_\mathbb G\bigg(H_2^B(f_1,f_2)\bigg)^q {\vert x \vert}^{\alpha}dx\bigg)^\frac{1}{q}\le  C\bigg\{\int_{\mathbb G} {\vert f_1(x) \vert}^{p_1}{\vert x \vert}^{\beta_1}dx\bigg\}^{\frac1{p_1}}\\&\quad\quad\quad\quad\quad\quad \quad\quad\quad\quad\times \bigg\{\int_{\mathbb G} {\vert f_2(x) \vert}^{p_2}{\vert x \vert}^{\beta_2}dx\bigg\}^{\frac1{p_2}}
	\end{align}
	holds for all measurable functions $f:\mathbb G\to{\mathbb C}$  if and only if $\beta_1(1-p'_1)+Q>0,$ $\alpha+Q<0,$ $\beta_2(1-p'_2)+Q>0,$  $\frac{(\alpha+Q)r_2}{q}+\frac{(\beta_2(1-p'_2)+Q)r_2}{q'}+\beta_2(1-p'_2)+Q<0$ and $\frac{\alpha+Q}{q}+\frac{\beta_1(1-p_1')+Q} {p_1'}+\frac{{\beta_2}(1-p_2')+Q} {p_2'}=0.$
\end{cor}
	\begin{rem}
		Similarly, the conditions for the remaining cases of Theorem \ref{THM:Hardy3} can be determined.
	\end{rem}

\subsection{Hyperbolic spaces} 
 
 Let $\mathbb H^n$ be the hyperbolic space of dimension $n$. Let $a\in \mathbb H^n$.
 Let us consider the weights 
 $$u(x)= (\sinh {\vert x \vert_a})^\alpha,\, v_1(x)= (\sinh {\vert x \vert_a})^{\beta_1}\textrm{ and }v_2(x)= (\sinh {\vert x \vert_a})^{\beta_2}.$$  
Then, $\mathcal D_1$  in terms of polar coordinates is equivalent to
 \begin{align}\label{condition:1}
 \nonumber&{\mathcal D_1} \simeq \sup_{\vert x \vert_a>0}\bigg(\displaystyle\int_{\vert x \vert_a}^\infty(\sinh{\rho})^{\alpha+n-1}d\rho\bigg)^\frac{1}{q}\bigg(\displaystyle\int_0^{\vert x \vert_a} (\sinh{\rho})^{\beta_1(1-p_1')+n-1}d\rho \bigg)^\frac{1}{p_1'}\\&\quad\quad\quad\quad\times\bigg(\displaystyle\int_0^{\vert x \vert_a} (\sinh{\rho})^{{\beta_2}(1-p_2')+n-1}d\rho \bigg)^\frac{1}{p_2'}.
 \end{align}
 To evaluate the first, second and the third integral, we need
$\alpha+n-1<0$,  $ \beta_1(1-p_1')+n>0,$ and $ {\beta_2}(1-p_2')+n>0,$ respectively.\\
Let us further examine the conditions for this supremum to be finite. 
For $\vert x \vert_a \gg1 $, we have $\sinh{\rho}\approx \exp {\rho}.$ Then \eqref{condition:1} can be stated as
\begin{align}
\nonumber&\sup_{\vert x \vert_a \gg1} \bigg(\displaystyle\int_{\vert x \vert_a}^\infty (\exp{\rho})^{\alpha+n-1}d\rho\bigg)^\frac{1}{q} \bigg(\displaystyle\int_0^{\vert x \vert_a} (\exp{\rho})^{\beta_1(1-p_1')+n-1}d\rho \bigg)^\frac{1}{p_1'}\\&\nonumber\quad\quad\quad\quad\quad\quad\quad\quad\times\bigg(\displaystyle\int_0^{\vert x \vert_a} (\exp{\rho})^{{\beta_2}(1-p_2')+n-1}d\rho \bigg)^\frac{1}{p_2'}\\&
\quad\quad\quad\simeq \sup_{\vert x \vert_a \gg1} (\exp {\vert x \vert_a})^{\bigg(\frac{\alpha+n-1}{q}+\frac{\beta_1(1-p_1')+n-1}{p_1'} +\frac{{\beta_2}(1-p_2')+n-1}{p_2'} \bigg)},\nonumber
\end{align} 
which is finite if and only if $$\frac{\alpha+n-1}{q}+\frac{\beta_1(1-p_1')+n-1}{p_1'} +\frac{{\beta_2}(1-p_2')+n-1}{p_2'}\le0.$$
For $\vert x \vert_a\ll 1 $, we have $\sinh{\rho}\approx\rho.$ Now, it can be represented as
\begin{align}
    \nonumber&\sup_{\vert x \vert_a\ll 1} \bigg(\displaystyle\int_{\vert x \vert_a}^\infty(\sinh{\rho})^{\alpha+n-1}d\rho\bigg)^\frac{1}{q} \bigg(\displaystyle\int_0^{\vert x \vert_a} {\rho}^{\beta_1(1-p_1')+n-1}d\rho \bigg)^\frac{1}{p_1'}\\&\nonumber\quad\quad\quad\quad\quad\quad\quad\quad\quad\times\bigg(\displaystyle\int_0^{\vert x \vert_a} {\rho}^{{\beta_2}(1-p_2')+n-1}d\rho \bigg)^\frac{1}{p_2'}
\\&\nonumber\simeq \sup_{\vert x \vert_a\ll 1} \bigg(\displaystyle\int_{\vert x \vert_a}^ R (\sinh{\rho})^{\alpha+n-1}d\rho +\displaystyle\int_{R}^\infty(\sinh{\rho})^{\alpha+n-1}d\rho\bigg)^\frac{1}{q} 
\\&\quad\quad\quad\times\bigg(\displaystyle\int_0^{\vert x \vert_a} {\rho}^{\beta_1(1-p_1')+n-1}d\rho \bigg)^\frac{1}{p_1'}\bigg(\displaystyle\int_0^{\vert x \vert_a} {\rho}^{\beta_2(1-p_2')+n-1}d\rho \bigg)^\frac{1}{p_2'}.\nonumber
\end{align}

 We have ${\sinh{\rho} }_{\vert x \vert_a\le \rho<R} \approx \rho$ for some small $R,$ and in that case the above supremum is 
\begin{align*}
    \approx \sup_{\vert x \vert_a\ll 1}\bigg( \vert x \vert_a^{\alpha+n}+C_R\bigg)^\frac{1}{q} {\vert x \vert_a }^\frac{\beta_1(1-p')+n}{p'_1}{\vert x \vert_a }^\frac{\beta_2(1-p_2')+n}{p_2'}.\end{align*}

Now, for $\alpha+n\ge0$, this is 
$$\approx \sup_{\vert x \vert_a\ll 1}{\vert x \vert_a}^{\frac{\beta_1(1-p_1')+n}{p_1'}+\frac{{\beta_2}(1-p_2')+n}{p_2'}},$$ which is finite if and only if $${\frac{\beta_1(1-p_1')+n}{p_1'}+\frac{{\beta_2}(1-p_2')+n}{p_2'}} \ge0.$$

At the same time, for $\alpha+n <0 $ it is

$$\approx \sup_{\vert x \vert_a\ll 1}{\vert x \vert_a}^{\frac{\alpha+n}{q}+{\frac{\beta_1(1-p_1')+n}{p_1'}+\frac{\beta_2(1-p_2')+n}{p_2'}}},$$  which is finite if and only if $$\frac{\alpha+n}{q}+{\frac{\beta_1(1-p_1')+n}{p_1'}+\frac{\beta_2(1-p_2')+n}{p_2'}}\ge0.$$

Summarising, we obtain the following result:

\begin{cor}\label{COR:hyp}
Let $a\in \mathbb H^n,$ where $\mathbb H^n$ be the hyperbolic space of dimension $n.$ The hyperbolic distance from $x$ to $a$ is denoted by  $|x|_a$. Let $1< \max(p_1,p_2) \leq q<\infty$ and let $\alpha,\beta_1, \beta_2 \in\mathbb R$. Then the inequality
\begin{align}\label{EQ:Hyp}
\nonumber&\bigg(\int_{\mathbb {H}_n} \bigg(H_2^B(f_1,f_2)\bigg)^q (\sinh {\vert x \vert_a})^\alpha dx\bigg)^\frac{1}{q}\le C\bigg\{\int_{\mathbb H^n} {\vert f_1(x) \vert}^{p_1} (\sinh {\vert x \vert_a})^{\beta_1} dx\bigg\}^{\frac1{p_1}}\\&\quad\quad\quad\quad\quad\quad\times\bigg\{\int_{\mathbb H^n} {\vert f_2(x) \vert}^{p_2} (\sinh {\vert x \vert_a})^{\beta_2} dx\bigg\}^{\frac1{p_2}}
\end{align}
holds for all measurable functions $f:\mathbb H^n\to{\mathbb C}$  if 
the parameters satisfy either of the following conditions:
\begin{itemize}
\item[(A)] for $\alpha+n \ge 0, $ if $\alpha+n<1$, $ \beta_1(1-p_1')+n>0$ and $ \beta_2(1-p_2')+n>0$ and $\frac{\alpha+n}{q}+\frac{\beta_1(1-p_1')+n}{p_1'}+\frac{\beta_2(1-p_2')+n}{p_2'} \le \frac{1}{q}+\frac{1}{p'_1}+\frac{1}{p'_2};$
\item[(B)] for $\alpha+n <0 $, if $ \beta_1(1-p_1')+n>0$ and $ \beta_2(1-p_2')+n>0$   and $0\leq \frac{\alpha+n}{q}+\frac{\beta_1(1-p_1')+n}{p_1'}+\frac{\beta_2(1-p_2')+n}{p_2'} \le \frac{1}{q}+\frac{1}{p'_1}+\frac{1}{p'_2}$.
\end{itemize} 
\end{cor} 
\begin{rem}
	The other cases of Theorem \ref{THM:Hardy3} can be treated in a similar way.
\end{rem}

\subsection{Cartan-Hadamard manifolds}

Let $(M,g)$ be a Cartan-Hadamard manifold and let $K_M$ be constant sectional curvature. Under this assumption, it is well known that $J(t,\omega)$ is a function of $t$ only.  To elaborate, if  $K_M=-b$ for $b\ge0$, then from \eqref{condition:11}, we have 
$J(t,\omega)= 1$ for $b=0$, and $J(t,\omega)=\left(\frac{\sinh \sqrt{b}t}{\sqrt{b}t}\right)^{n-1}$ for $b>0$,
see e.g. \cite {BC, GHL}.
 
 When $b=0$, let us take $u(x)=  \vert x \vert_a^\alpha,$ $v_1(x)=  \vert x \vert_a^\beta$ and $v_2(x)=  \vert x \vert_a^\gamma$ , then the  inequality \eqref{EQ:Hardy1} holds for $1< \max(p_1,p_2) \leq q<\infty$ if and only if
 \begin{align}
     \nonumber&\sup_{\vert x \vert_a>0} \bigg(\displaystyle\int_{M \backslash B(a,\vert x \vert_a)}\vert y \vert_a^\alpha dy \bigg)^\frac{1}{q} \bigg(\displaystyle \int_{B(a,\vert x \vert_a)}\vert y \vert_a^{\beta(1-p_1')} dy\bigg)^\frac{1}{p_1'}\\&\quad\quad\quad\quad\bigg(\displaystyle \int_{B(a,\vert x \vert_a)}\vert y \vert_a^{\gamma(1-p_2')} dy\bigg)^\frac{1}{p_2'}<\infty.\end{align}

After changing to the polar coordinates (see \eqref{condition:11}), this is equivalent to            

$$\sup_{\vert x \vert_a>0}\bigg(\int_{\vert x \vert_a }^\infty  {\rho}^{\alpha+n-1} d\rho \bigg)^\frac{1}{q}\bigg(\int_0^{\vert x \vert_a} {\rho}^{\beta(1-p_1')+n-1}d\rho \bigg)^\frac{1}{p_1'}\bigg(\int_0^{\vert x \vert_a} {\rho}^{\gamma(1-p_2')+n-1}d\rho \bigg)^\frac{1}{p_2'},$$ 
which is finite if and only if conditions of Corollary \ref{COR:hom} hold with $Q=n$ (which is natural since the curvature is zero).

When $b>0$, let us take $u(x)=(\sinh \sqrt{b}{\vert x \vert_a})^\alpha,$ $v_1(x)=(\sinh \sqrt{b}{\vert x \vert_a})^\beta$  and $v_2(x)=(\sinh \sqrt{b}{\vert x \vert_a})^\gamma$. Then the  inequality \eqref{EQ:Hardy1} holds for $1< \max(p_1,p_2) \leq q<\infty$ if and only if
\begin{align*}
    &\sup_{\vert x \vert_a>0} \bigg(\displaystyle\int_{ M \backslash B(a,\vert x \vert_a)} (\sinh \sqrt{b}{\vert y \vert_a})^{\alpha} dy \bigg)^\frac{1}{q} \bigg(\displaystyle \int_{B(a,\vert x \vert_a)}(\sinh \sqrt{b}{\vert y \vert_a})^{\beta(1-p_1')}dy \bigg)^\frac{1}{p_1'}\\&
\quad\quad\quad\quad\quad\quad\times \bigg(\displaystyle \int_{B(a,\vert x \vert_a)}(\sinh \sqrt{b}{\vert y \vert_a})^{\gamma(1-p_2')}dy \bigg)^\frac{1}{p_2'}<\infty.\end{align*}

After changing to the polar coordinates, this supremum is equivalent to 
\begin{align*}
&\sup_{\vert x \vert_a>0}\bigg(\int_{\vert x \vert_a}^\infty (\sinh \sqrt{b}{t})^\alpha (\frac{\sinh \sqrt{b} t}{\sqrt{b} t})^{n-1} t^{n-1}dt \bigg)^\frac{1}{q}\\&\quad\quad\quad\quad\quad\times\bigg( \int_{0}^ {\vert x \vert_a} (\sinh \sqrt{b}{t})^{\beta(1-p_1')}(\frac{\sinh \sqrt{b} t}{\sqrt{b} t})^{n-1} t^{n-1} dt \bigg)^\frac{1}{p_1'}\\&\quad\quad\quad\quad\quad\times
\bigg( \int_{0}^ {\vert x \vert_a} (\sinh \sqrt{b}{t})^{\gamma(1-p_2')}(\frac{\sinh \sqrt{b} t}{\sqrt{b} t})^{n-1} t^{n-1} dt \bigg)^\frac{1}{p_2'}
\\&\simeq\sup_{\vert x \vert_a>0}\bigg(\displaystyle\int_{\vert x \vert_a}^\infty (\sinh \sqrt{b}{t})^{\alpha+n-1}dt \bigg)^\frac{1}{q}\bigg(\displaystyle\int_0^{\vert x \vert_a} (\sinh \sqrt{b} t)^{\beta(1-p_1')+n-1}dt \bigg)^\frac{1}{p_1'}\\& 
\quad\quad\quad\quad\quad\times\bigg(\displaystyle\int_0^{\vert x \vert_a} (\sinh \sqrt{b} t)^{\gamma(1-p_2')+n-1}dt \bigg)^\frac{1}{p_2'},
\end{align*}
which has the same conditions for finiteness as the case of the hyperbolic space in Corollary \ref{COR:hyp} (which is also natural since it is the negative constant curvature case).

\section{Acknowledgement}
During the start of this work, the third author (DV) visited the Ghent Analysis \& PDE Centre, Ghent University, Belgium, as a visiting faculty from Miranda House, University of Delhi. The second (AS) and third author (DV) acknowledge the Ghent Analysis \& PDE Centre, Ghent University, Belgium, for their financial support and warm hospitality during the research visit. The first author (MR) is supported by the FWO Odysseus 1 grant G.0H94.18N: Analysis and Partial
Differential Equations and by the Methusalem programme of the Ghent University Special
Research Fund (BOF) (Grant number 01M01021). MR is also
 supported by the FWO Senior Research Grant G011522N and by EPSRC grant EP/V005529/1.

\end{document}